\documentclass[imssubmit,noinfoline]{imsart}
\RequirePackage[OT1]{fontenc}
\RequirePackage{amssymb, amsfonts, amsmath, amsthm}
\usepackage[active]{srcltx}
\usepackage{hyperref}
\usepackage[dvips]{color}
\usepackage{multicol}

\startlocaldefs

\def\mathink#1{#1}
\def\r{\sigma}
\def\fx#1#2{\left(\frac{#1}{#2}-1\right)}

\numberwithin{equation}{section}
\theoremstyle{plain}

\newtheorem{theorem}{Theorem}[section]
\newtheorem{prop}{Proposition}[section]
\newtheorem{lemma}{Lemma}[section]
\newtheorem{cor}{Corollary}[section]

\def\SECT{\section}
\def\SSECT{\subsection}

\def\Grp#1{\left(#1\right)}
\def\Cbr#1{\left\{#1\right\}}
\def\Sbr#1{\left[#1\right]}
\def\Abs#1{\left|#1\right|}

\def\nth#1{\frac{1}{#1}}
\def\cf#1{\mathbf{1}\Cbr{#1}}
\def\cum#1#2{{#1}_1+\cdots+{#1}_{#2}}
\def\inline#1{\mbox{$#1$}}

\def\Reals{\mathbb{R}}

\def\toi{\to\infty}
\def\Linf{\mathop{\underline{\mathrm{lim}}}}
\def\Lsup{\mathop{\overline{\mathrm{lim}}}}

\endlocaldefs

\begin{document}
\begin{frontmatter}
  \title{On the asymptotic of likelihood ratios for self-normalized
    large deviations}

  \runtitle{Likelihood ratio of self-normalized LDP}

  \begin{aug}
    \author{
      \fnms{Zhiyi}
      \snm{Chi}
      \thanksref{nih}
      \ead[label=email]{zchi@stat.uconn.edu}
    }
    \thankstext{nih}{
      Research partially supported by NSF grant DMS-0706048 and NIH
      grant DC007206-01.
    }

    \runauthor{Z. Chi}
    
    \affiliation{
      Department of Statistics, University of Connecticut \\
      \rule{0em}{1em}
    }

    \address{
      Department of Statistics \\
      University of Connecticut \\
      215 Glenbrook Road, U-4120 \\
      Storrs, CT 06269 \\
      \printead{email}
    }
  \end{aug}

  \begin{abstract}
    Motivated by multiple statistical hypothesis testing, we obtain
    the limit of likelihood ratio of large deviations for
    self-normalized random variables, specifically, the ratio of
    $P(\mathink{\sqrt{n}(\bar X +d/n)} \ge x_n V)$ to
    $P(\mathink{\sqrt{n}\bar X} \ge x_n V)$, as $n\toi$, 
    where $\bar X$ and $V$ are the sample mean and standard deviation
    of iid $X_1, \ldots, X_n$, respectively, $d>0$ is a constant and
    $x_n \toi$.  We show that the limit can have a simple form
    $e^{d/z_0}$, where $z_0$ is the unique maximizer of $z f(x)$ with
    $f$ the density of $X_i$.  The result is applied to derive the
    minimum sample size per test in order to control the error rate of
    multiple testing at a target level, when real signals are
    different from noise signals only by a small shift.
  \end{abstract}


  
\end{frontmatter}

\SECT{Introduction}

\SSECT{Background}
Suppose $X_1$, $X_2$, \ldots are iid random variables with density
$f$, such that $P(X_1>0)>0$.  For $n\ge 1$, let
$S_n = \cum X n$.  We shall consider the biased $t$ statistic
\begin{align*}
  T_n = \frac{\sqrt{n} \bar X}{V}, \ \ \text{ with }\ 
  \bar X = \frac{S_n}{n},\ \ V = \Sbr{\nth
    n\sum_{i=1}^n (X_i - \bar X)^2}^{1/2}.
\end{align*}
The choice for $T_n$ is only for simplicity of notation.  All the
results obtained for $T_n$ in the paper hold for the standard $t$
statistic $\sqrt{n-1}\bar X/V$ as well.

The aim here is to find the limit of the ratio of tail probabilities
for $T_n$, specifically, the limit of
\begin{align*}
  \frac{P\Grp{\mathink{\sqrt{n}(\bar X + d/n)} \ge x_n V}}{
    P\Grp{\mathink{\sqrt{n}\bar X} \ge x_n V}}, \qquad\text{as $n\toi$},
\end{align*}
where $d>0$ is a constant and $x_n\toi$ in a suitable rate.  The
problem pertains to large deviations for self-normalized random
variables \cite{dembo:shao:06, shao:97}.  On the other hand, it is
directly related to statistical multiple hypothesis testing, in
particular, the False Discovery Rate (FDR) control
\cite{benjamini:hoc:95}, which in recent years has generated intensive
research due to its applications in microarray data analysis, medical
imagery, etc, where a very large number of signals (``null
hypotheses'') have to be sorted through in order to identify signals
of interest (``false nulls'') from the other, noise signals (``true
nulls'') \cite{efron:etal:01, genovese:was:02,
  pacifico:gen:ver:was:04, storey:tay:sie:04}. 

A measure of performance for multiple testing is the fraction of
falsely identified noise signals (``false discoveries'') among the
identified ones.  Given that at least one signal is identified, the
fraction is a well-defined random variable and its conditional
expectation is called positive FDR, or pFDR.  For a testing procedure,
it is desirable that, given a target control level $\alpha$, the
procedure attains pFDR $\le \alpha$.  However, whether or not this is
possible depends on the property of the data distributions as well as
how much data is available to assess the hypotheses.  We consider a
typical multiple testing problem, where the data distributions are
shifted and scaled versions of each other.

Suppose the data distributions are $F_i(x) = F(s_i x - u_i)$, where
$F$ is a fixed distribution, and $s_i>0$ and $u_i$ are unknown.  In
order to identify from $F_i$ those with $u_i \not= 0$, we test null
(hypotheses) $H_i: u_i=0$ to see which one can be rejected.  To this
end, let $n$ iid observations be sampled from $F_i$, which can be
written as $Y_{i1} = (X_{i1}+u_i) / s_i$, \ldots, $Y_{in} = (X_{in} +
u_i)/s_i$, with $X_{ij} \sim F$.  Suppose the nulls are tested
independently of each other, so that $X_{ij}$ are iid for $i\ge 1$,
$j=1,\ldots,n$.  Typically, $H_i$ is rejected if and only if the $t$
statistic of $Y_{i1}, \ldots, Y_{in}$ is larger than a cut-off value
$x_n$.  Suppose that false nulls occur randomly in the population of
nulls, such that each $H_i$ can be false with probability $p\in (0,1)$
independently of the others, and $u_i=u>0$ when $H_i$ is false.  By
definition, a falsely rejected null is a true null, i.e., $u_i=0$.
It is then not hard to see
\begin{align}
  P(H_i \text{ is falsely rejected}\,|\, H_i \text{ is rejected}) 
  = \frac{1-p}{1-p + p R_n},  \label{eq:post}
\end{align}
where $R_n$ is the ratio of tail probabilities
\begin{align*}
  R_n(u) = \frac{P(\sqrt{n}(\bar X + u) \ge x_n V)}{
    P(\sqrt{n}\bar X\ge x_n V)}.
\end{align*}

It follows that the minimum attainable pFDR is equal to the right hand
side of \eqref{eq:post} as well \cite{chi:tan:07}.  Consequently, if
real signals are weak in the sense that $u\approx 0$, then $R_n$ can
be close to 1, implying that when a nonempty set of nulls are
rejected by whatever multiple testing procedure, it is likely that
most or almost all of them are falsely rejected.

For the $t$ test, the only way to address the above limitation on the
error rate control is to increase $n$, the number of observations for
each null.  From \eqref{eq:post}, in order to attain pFDR $\le
\alpha$, $n$ must satisfy
\begin{align}
  R_n(u) \ge (1/p-1)(1/\alpha-1).  \label{eq:pfdr}
\end{align}
An important question is, as $u\approx 0$, what would be the minimum
$n$ in order for \eqref{eq:pfdr} to hold.

The issue of sample size for pFDR control was previously studied in
\cite{chi:07b}.  However, in that work the $t$ statistic was defined
in a different way, with $\bar X$ and $V$ derived from two independent
samples instead of from the same sample.  Although that definition
allows an easier treatment, it is not commonly used in practice.
Furthermore, the asymptotic result in \cite{chi:07b} is different from
the one reported here for the more commonly used $t$ statistic.

\SSECT{Main results}
We need to be more specific about the cut-off value $x_n$.  Usually,
as $n$ increases, one can afford to look at more extreme tails to get
stronger evidence against nulls.  This suggests there should be
$x_n\toi$ as $n\toi$.  If $EX>0$ and $EX^2 < \infty$ for $X\sim F$,
then $x_n$ should be at least of the same order as $\sqrt{n}$,
otherwise $\inf\text{pFDR}\to 1$, where the infimum is taken over all
possible multiple testing procedures that are solely based on $T_i$.
Furthermore, for $F=N(0,1)$, it is known that there should be
$x_n/\sqrt{n}\toi$ in order to attain $\inf\text{pFDR}$
\cite{chi:07b}.  Based on the considerations, for the general case, we
will impose $x_n=a_n\sqrt{n}$ with $a_n\toi$ as the cut-off value.

\begin{theorem} \label{thm:self-norm}
  Suppose the density $f$ satisfies the following conditions.
  \begin{itemize}
    \item[1)] $f$ is bounded and continuous on $\Reals$ and there
      is $\gamma>0$, such that 
      \begin{align*}
        \Lsup_{x\toi} x^{1+\gamma} f(x)<\infty.
      \end{align*}
    \item[2)] $zf(z)$ has a unique maximizer $z_0>0$. 
    \item[3)] $h := \log f$ is three times differentiable on $\Reals$,
      such that $\sup |h''|<\infty$ and $\sup |h'''|<\infty$.
    \end{itemize}
    Let $a_n\toi$, such that $a_n^4 = o(n/\log n)$.  Then for
    any $d_n\to d \in (0,\infty)$,
    \begin{align*}
      \frac{P\Grp{\bar X + d_n/n \ge \mathink{a_n V}}}{
        P\Grp{\bar X \ge \mathink{a_n V}}
      }
      \to e^{d/z_0}, \quad\text{as}\ \ n\toi.
    \end{align*}
    Note that for different $n$, $\bar X$ and $V$ are different random
    variables.
\end{theorem}

Let $k_*=k_*(u)$ be the minimum $n$ in order for \eqref{eq:pfdr}
to hold.  The asymptotic of $k_*$  as $u\to 0$ is a consequence of 
Theorem \ref{thm:self-norm}.
\begin{cor} \label{cor:minn}
  Suppose $f$ and $a_n$ satisfy the conditions in Theorem
  \ref{thm:self-norm}.  Let $p\in (0,1)$ and $\alpha\in (0,1)$ be
  fixed in \eqref{eq:pfdr}.  Then
  \begin{align*}
    k_*(u) \sim (z_0/u) \ln [(1/p-1)(1/\alpha-1)], \qquad
    \text{as}\ \ u\to 0+.
  \end{align*}
\end{cor}

Many probability densities satisfy conditions 1)--3) of Theorem
\ref{thm:self-norm}, for example, Gaussian density $f_1(x; \mu, \r) =
e^{-(x-\mu)^2/2\r^2}/\sqrt{2\pi}\r$ and Cauchy density $f_2(x; \mu,
\r) = \r\pi^{-1}[\r^2+(x-\mu)^2]^{-1}$.  In particular, when $\mu=0$
and $\r=1$, both have $z_0=1$.  Therefore, even though all the moments
of $f_1$ are finite whereas all those of $f_2$ are infinite, in terms
of the amount of data needed to control the pFDR, these two are
asymptotically the 
same.  On the other hand, Theorem \ref{thm:self-norm} is not
applicable to densities with zeros on $\Reals$.  Since the conclusion
of Theorem \ref{thm:self-norm} has nothing to do with the continuity
of $h = \log f$ over $\Reals$, it is desirable to remove condition~3)
altogether.

In the rest of the paper, Section~2 proves Theorem \ref{thm:self-norm}
and Corollary \ref{cor:minn}.  Sections~3 and 4 contain proofs of
lemmas for the main results.

\SECT{Proof of main results}
A key to the proof is the fact that the analysis can be localized at
$z_0$, which is revealed by a representation of the event $\{T_n\ge
\sqrt{n} a_n\}$ given by Shao \cite{shao:97}.  It is easily seen that
for $t>0$,
\begin{align*}
  &
  \Cbr{T_n \ge t} = \Cbr{\frac{S_n}{Q_n} \ge
    t\Grp{\frac{n}{n+t^2}}^{1/2}}, \\
  \text{where}\quad&
  Q_n = \inline{\sqrt{X_1^2+\cdots+X_n^2}}\,,
\end{align*}
(cf.\ \cite{shao:97}).  If $t = \sqrt{n} a_n$, then, letting $r = 1-
(1+a_n^{-2})^{-1/2}$ and following \cite{shao:97},
\begin{align*}
  \Cbr{T_n \ge \sqrt{n}a_n}
  &
  = \Cbr{\frac{S_n}{Q_n\sqrt{n}} \ge 1-r} \\
  &=
  \Cbr{\sup_{b>0} \sum_{i=1}^n \Sbr{b X_i - \frac{(1-r)}{2}
      (X_i^2 + b^2)}\ge 0} \\
  &=
  \Cbr{\sup_{b>0} \sum_{i=1}^n \Sbr{
      \frac{b^2r(2-r)}{2(1-r)} - \frac{1-r}{2}\Grp{
        X_i - \frac{b}{1-r}
      }^2}\ge 0}\\
  &=
  \Cbr{\sup_{b>0} \sum_{i=1}^n \Sbr{
      \frac{b^2r(2-r)}{(1-r)^2} - \Grp{
        X_i - \frac{b}{1-r}
      }^2}\ge 0}\,.
\end{align*}
Let $z=b/(1-r)$ and $\r_n = \sqrt{r(2-r)}$.  Then
\begin{align}
  \Cbr{T_n \ge \sqrt{n} a_n}
  &
  = \Cbr{
    \r_n^2\ge \inf_{z>0}\nth n\sum_{i=1}^n \fx{X_i}{z}^2
  }.   \label{eq:t-local}
\end{align}

Under the assumption of Theorem \ref{thm:self-norm}, $r = 1-
(1+a_n^{-2})^{-1/2} = a_n^{-2}/2 + o(a_n^{-2})$, and hence $\r_n^2
\sim 2 r = a_n^{-2} + o(a_n^{-2})$, yielding
\begin{align}
  \r_n\to 0, \quad n\r_n^4/\log n \sim n/(a_n^4 \log n)\toi.
  \label{eq:var}
\end{align}
Equations \eqref{eq:t-local} and \eqref{eq:var} are the starting
point of the proof.

\begin{lemma}\label{lemma:local}
  Suppose $f$ satisfies condition 1) and 2) in Theorem
  \ref{thm:self-norm}.  Let $\r_n\to 0$  such that $n\r_n^4/\log n\toi$.
  Then, given $r>0$, there is $\delta = \delta(r)>0$, such that
  \begin{align*}
    \lim_{n\toi} \sup_{|d|\le \delta}
    \frac{\displaystyle
      P\Cbr{\r_n^2 \ge \inf_{z>0} \nth n \sum_{i=1}^n \fx{X_i+d}{z}^2}
    }{\displaystyle
      P\Cbr{\r_n^2 \ge \inf_{|z-z_0|\le r} \nth n \sum_{i=1}^n
        \fx{X_i+d}{z}^2}
    } = 1.
  \end{align*}
\end{lemma}

The lemma will be proved later.  The following heuristic explains why
the analysis can be localized at $z_0$.  Let $d=0$.  For
$\r_n^2\ll 1$, if the event $E_z=\{\r_n^2\ge (1/n)\sum_{i=1}^n
(X_i/z-1)^2\}$ occurs, then most of $X_i$ must fall between
$(1-\r_n)z$ and $(1+\r_n)z$, implying
\begin{align*}
  \log P(E_z) \approx n \log P(|X-z|\le \r_n z) \approx n \log
  (2\r_n z f(z)).
\end{align*}
As a result, given that at least one $E_z$ occurs, the most likely
value of $z$ should be the maximizer of $z f(z)$, i.e., $z_0$.

The following fact will be used in the proof of Theorem
\ref{thm:self-norm}.  If $X_1,\ldots, X_n$ are iid with density $f$
and $n\ge 3$, then the joint density of $\bar X$ and $V$ is
\begin{align}
  h(t,s) = (\sqrt{n})^n s^{n-2} \int \prod_{i=1}^n f(t+\sqrt{n} s
  \omega_i) \mu_n(d\omega)  \label{eq:sphere}
\end{align}
where $\mu_n$ is the uniform distribution on a $(n-2)$ dimensional
unit sphere perpendicular to $(1,1,\ldots, 1)$ in $\Reals^n$, i.e.,
\begin{align*}
  U_n:=\Cbr{\omega\in\Reals^n: \sum_{i=1}^n\omega_i^2=1,\
    \sum_{i=1}^n \omega_i = 0}\,.
\end{align*}
For completeness, a sketch of the proof of \eqref{eq:sphere} is given
in the Appendix.

Finally, recall that for any $a\in \Reals$ and random variables
$\xi_1, \ldots, \xi_n$,
\begin{align*}
  \nth n\sum_{i=1}^n (\xi_i - a)^2 = (\bar
\xi - a)^2 + V_\xi^2,
\end{align*}
where $\bar\xi$ is the sample mean of $\xi_i$, and $V_\xi = n^{-1/2}
\sqrt{\sum_{i=1}^n (\xi_i - \bar\xi)^2}$ is the biased sample standard
deviation.

\SSECT*{\sc Proof of Theorem \ref{thm:self-norm}}
Fix $d_n\ge 0$ such that $d_n\to d<\infty$.  Given $r>0$, for $n\gg
1$, $|d_n/n|\le\delta$, where $\delta = \delta(r)>0$ is as in Lemma
\ref{lemma:local}.  It therefore suffices to consider the limit of
\begin{align*}
  L_n:=\frac{\displaystyle
    P\Cbr{\r_n^2 \ge \inf_{|z-z_0|\le r} \nth n \sum_{i=1}^n
      \fx{X_{i,n}}{z}^2}
  }{\displaystyle
    P\Cbr{\r_n^2 \ge \inf_{|z-z_0|\le r} \nth n \sum_{i=1}^n
      \fx{X_i}{z}^2}
  }
\end{align*}
where $X_{i,n} := X_i+d_n/n$ has density $f(x-d_n/n)$.  Let
\begin{align*}
  \Gamma_n = \Cbr{(t,s) \in (-\infty,\infty)\times [0,\infty):
    \r_n^2\ge \inf_{|z-z_0|\le r} \frac{(t-z)^2 + s^2}{z^2}}.
\end{align*}
Then for any random variables $\xi_1, \ldots, \xi_n$,
\begin{align*}
  \Cbr{
    \r_n^2 \ge \inf_{|z-z_0|\le r}
    \nth n \sum_{i=1}^n\fx{\xi_i}{z}^2
  }
  =
  \Cbr{(\bar\xi, V_\xi)\in \Gamma_n}.
\end{align*}

Apply the above formula to $X_{i,n}$ and $X_i$ respectively.  By
\eqref{eq:t-local} and \eqref{eq:sphere},
\begin{align*}
  L_n&
  = 
  \frac{
    \displaystyle
    \int_{(t,s, \omega)\in\Gamma_n\times U_n}
    s^{n-2}\prod_{i=1}^n
    f(t-d_n/n+\sqrt{n} s \omega_i)\,\mu_n(d\omega)\,d t\,d s
  }{
    \displaystyle
    \int_{(t,s, \omega)\in\Gamma_n\times U_n}
    s^{n-2}\prod_{i=1}^n
    f(t+\sqrt{n} s \omega_i)\,\mu_n(d\omega)\,d t\,d s
  } \\
  &
  =
  \int_{(t,s, \omega)\in\Gamma_n\times U_n}
  \rho(t,s,\omega) \nu(d t,d s,d\omega),
\end{align*}
where $\nu(dt,ds,d\omega)$ is the probability measure on $\Gamma_n
\times U_n$ proportional to $s^{n-2}\prod_{i=1}^n f(t+\sqrt{n} s
\omega_i)\,\mu_n(d\omega)\,dt\,ds$, and
\begin{align*}
  \rho(t,s,\omega)
  =\frac{\prod_{i=1}^n f(t-d_n/n+\sqrt{n} s \omega_i)}{
    \prod_{i=1}^n f(t+\sqrt{n} s \omega_i)}.
\end{align*}
For each $(t,s,\omega)\in\Gamma_n\times U_n$, by Taylor expansion,
\begin{align*}
  \rho(t,s,\omega)&
  = \exp\Cbr{
    \sum_{i=1}^n \Sbr{
      h(t+\sqrt{n}s\omega_i - d_n/n) - h(t+\sqrt{n}s\omega_i)
    }
  } 
  \\
  &
  = \exp\Cbr{
    -\frac{d_n}{n}\sum_{i=1}^n h'(t+\sqrt{n}s\omega_i) + e_n}
\end{align*}
where $\sup_{(t,s,\omega)} |e_n| = O(d_n^2/n) = O(1/n)$ due to
$\sup_x |h''(x)|<\infty$.   By Taylor expansion and $\cum \omega
n = 0$,
\begin{align*}
  \frac{1}{n}\sum_{i=1}^n h'(t+\sqrt{n}s\omega_i) 
  = h'(t) + \nth n \sum_{i=1}^n h'''(t+\theta \sqrt{n} s \omega_i)
  (\sqrt{n} s\omega_i)^2
\end{align*}
for some $\theta \in (0,1)$.  Because $\omega_i^2$ add up to 1 and
$(t,s)\in \Gamma_n$,
\begin{align*}
  \Abs{\frac{1}{n}\sum_{i=1}^n h'(t+\sqrt{n}s\omega_i) 
    - h'(t)} \le \sup_x |h'''(x)| s^2 \le A\r_n^2,
\end{align*}
where $A = (z_0+r)^2\sup_x |h'''(x)| <\infty$.  For $(t,s)\in
\Gamma_n$, as $|t-z|\le \r_n z$ for some $z\in 
[z_0-r, z_0+r]$, $|t-z_0|\le r+\r_n(z_0+r) < 2r$ for $n\gg 1$.
Combining the above 
bounds,
\begin{align*}
  e^{-d_n\Delta(2r) - A\r_n^2 - \sup |e_n|}\le
  \frac{\rho(t,s,\omega)}{e^{-d_n h'(z_0)}} \le 
  e^{d_n\Delta(2r) + A\r_n^2 +\sup |e_n|},
\end{align*}
with $\Delta(c) = \sup_{|z-z_0|\le c}  |h'(z)-h'(z_0)|$.  Since
$r$ is arbitrary and $h'$ is continuous, from the expression of
$L_n$ and $d_n\to d$, it is seen that $L_n\sim e^{-dh'(z_0)}$ as
$n\toi$.  Finally, since $z_0$ maximizes $\log z + h(z)$, $h'(z_0) =
-1/z_0$.  So $L_n \sim e^{d/z_0}$.  \qed

\SSECT*{\sc Proof of Corollary \ref{cor:minn}}
First, it is necessary to show that as $u\to 0+$, $k_*(u)\toi$.   To
this end, it suffices to show that, when $n$ and $c>0$ are fixed, then
\begin{align*}
  \ell(u):=\frac{P(\bar X + u\ge c V)}{P(\bar X\ge c V)} \to 1,
  \qquad
  \text{as}\ \ u\to 0+,
\end{align*}
where $\bar X$ and $V$ are defined in terms of $X_1, \ldots, X_n$.
The limit follows from a corollary to Fatou's lemma, which states that
if $l_n(x)\le f_n(x)\le u_n(x)$, $l_n(x)\to l(x)$, $f_n(x)\to f(x)$
and $u_n(x)\to u(x)$ pointwise as $n\toi$, and $\int l_n\to \int l$
and $\int u_n\to \int u$, then $\int f_n\to \int f$.  Specifically, let
\begin{align*}
  A(r) = \Cbr{(x_1, \ldots, x_n): r^2\ge \inf_{z>0} \nth n
    \sum_{i=1}^n \fx{x_i}{z}^2}, \qquad\text{for}\ \ r>0.
\end{align*}
Then by \eqref{eq:t-local}, there is $\r\in (0,1)$, such that
\begin{align*}
  P(\bar X + u\ge c V)&
  = 
  \int \cf{x\in A(\r)} \prod_{i=1}^n f(x_i-u)\,d x_1\cdots d x_n \\
  P(\bar X\ge c V) &
  =
  \int \cf{x\in A(\r)} \prod_{i=1}^n f(x_i)\,d x_1\cdots d x_n.
\end{align*}

Apparently, $0\le \cf{x\in A(\r)} \prod_{i=1}^n f(x_i-u) \le
\prod_{i=1}^n f(x_i-u)$, with the right hand side having the
same integral as $\prod_{i=1}^n f(x_i)$.  Since $f(x-u)\to f(x)$
pointwise as $u\to 0$, the above corollary to Fatou's lemma implies
$P(\bar X + u\ge cV)\to P(\bar X\ge cV)>0$.  Then $\ell(u)\to 1$.

Next, we show that $u k_*(u)$ is bounded from $\infty$ as $u\to
0+$.  Suppose that there is a sequence $u_i$ such that $u_i
k_*(u_i)\to \infty$.  Clearly, $n_i:= k_*(u_i)\toi$.  Then, given
any $M$, $u_i n_i\ge M$ for $i\gg 1$ and hence by Theorem
\ref{thm:self-norm},
\begin{align*}
  \frac{
    P(\bar X + u_i \ge
    \mathink{a_{n_i}V})
  }{
    P(\bar X \ge \mathink{a_{n_i}V})
  }
  &
  \ge
  \frac{
    P(\bar X + M/n_i \ge \mathink{a_{n_i}V})
  }{
    P(\bar X \ge \mathink{a_{n_i}V})
  } \\
  &
  \to e^{M/z_0} \gg (1/p-1)(1/\alpha-1),
\end{align*}
which contradicts the definition of $k_*(u_i)$.

It only remains to show that $u k_*(u)\to d_0:=z_0
\ln[(1/p-1)(1/\alpha-1)]$ as $u\to 0$.  It suffices to show that for
any sequence $u_i\to 0$ with convergent $u_i k_*(u_i)$, the limit of
$u_i k_*(u_i)$ is $d_0$.  Indeed, let the limit be $d$.  Then,
following the above argument, $e^{d/z_0}=(1/p-1)(1/\alpha-1)$, giving
$d = d_0$.
\qed

\SECT{Proof of Lemma \ref{lemma:local}}
\begin{lemma}\rm \label{lemma:scale}
  Let $\r\in (0,1)$, $\eta>0$ and $s>0$.  Then
  \begin{align*}
    &
    \Cbr{
      \inf_{s\le z\le (1+\eta\r)s} \nth n\sum_{i=1}^n
      \fx{X_i}{z}^2 \le \r^2
    }\\
    \subset\ &
    \Cbr{
      \nth n \sum_{i=1}^n \fx{X_i}{s}^2 \le (1+\eta\r)^2
      (1+\eta)^2 \r^2
    }.
  \end{align*}
\end{lemma}

\begin{proof}
  Suppose $(1/n)\sum_{i=1}^n (X_i/z-1)^2 \le \r^2$ for some  $z\in [s,
  (1+\eta\r)s]$.  Then $|\bar X/z-1| \le \r$.  By $0\le 1-s/z \le
  \eta\r$ and $z^2/s^2 \le (1+\eta\r)^2$,
  \begin{align*}
    &\hspace{-2em}
    \nth n \sum_{i=1}^n \fx{X_i}{s}^2
    = \nth{n s^2} \sum_{i=1}^n (X_i - \bar X)^2 + \fx{\bar X}{s}^2
    \nonumber \\
    &
    = \frac{z^2}{s^2} \Sbr{
      \nth{n z^2} \sum_{i=1}^n (X_i - \bar X)^2 + 
      \Grp{\frac{\bar X}{z}-\frac{s}{z}}^2 
    }
    \nonumber \\
    &
    \le\frac{z^2}{s^2} \Sbr{
      \nth{n} \sum_{i=1}^n \fx{X_i}{z}^2
      + 2\Abs{\frac{\bar X}{z}-1}\Grp{1-\frac{s}{z}}
      + \fx{s}{z}^2
    },
  \end{align*}
  with the last expression no greater than $(1+\eta\r)^2
  (1+\eta)^2\r^2$.
\end{proof}

In the next, let $X_1, X_2, \ldots$ be iid random variables with
density $f$.
\begin{lemma} \rm \label{lemma:upper}
  Suppose $\Lsup_{x\toi} x^{1+\gamma} f(x)<\infty$ for some
  $\gamma>0$.  Let $\r_n\to 0$ such that $\Linf_n n\r_n>0$.  Then,
  given $T>0$ and $\delta>0$, there is $a=a(T,\delta)>0$, such that
  for $n\gg 1$,
  \begin{align*}
    \sup_{|d|\le \delta}\nth n \log P\Cbr{
      \r_n^2\ge \inf_{z\ge a} \nth n \sum_{i=1}^n \fx{X_i+d}{z}^2
    } \le \log\r_n - T.
  \end{align*}
\end{lemma}
\begin{proof}
  We first show that there is $a=a(T)>0$, such that
  \begin{align}
    \nth n \log P\Cbr{
      \r_n^2\ge \inf_{z\ge a} \nth n \sum_{i=1}^n \fx{X_i}{z}^2
    } \le \log\r_n - T.  \label{eq:upper-0}
  \end{align}

  Fix $\eta\in (0,1)$ with $\eta > (1+\eta/8)^2(1+\eta/4)^2-1$.  Let
  $\alpha_n = 1+\eta\r_n/4$.  
  For $n\ge 1$ with $\r_n<1/2$, $\alpha_n < 1+\eta/8$, so by Lemma
  \ref{lemma:scale}, for any $a>0$,
  \begin{align}
    &\quad
    P\Cbr{\r_n^2 \ge \inf_{z\ge a} \nth n \sum_{i=1}^n \fx{X_i}{z}^2}
    \nonumber\\
    &\le
    \sum_{j=0}^\infty
    P\Cbr{\r_n^2 \ge \inf_{a\alpha_n^j \le z \le
        a\alpha_n^{j+1}} \nth n \sum_{i=1}^n \fx{X_i}{z}^2}
    \nonumber\\
    &\le
    \sum_{j=0}^\infty P\Cbr{
      (1+\eta)\r_n^2 \ge \nth n \sum_{i=1}^n\fx{X_i}{a\alpha_n^j}^2
    }.  \label{eq:upper-1}
  \end{align}
  
  Let $s = (1+\eta) \r_n^2$.  By Chernoff's inequality, for $z>0$
  and $t>0$,
  \begin{align}
    P\Cbr{
      s \ge \nth n\sum_{i=1}^n \fx{X_i}{z}^2
    } 
    \le \Sbr{
      z\int e^{ts - t u^2} f(z+z u)\,d u
    }^n.  \label{eq:upper-2}
  \end{align}
  
  Fix $A > \Lsup_{x\toi} x^{1+\gamma}f(x)$.  Let $M(z) =
  (\gamma/2)\log z$ and $t=M(z)/s$.  Then
  \begin{align*}
    z\int e^{ts - t u^2} f(z+z u)\,d u&
    \le
    \frac{A}{z^\gamma (1-\eta)^{1+\gamma}}
    \int_{-\eta}^{\eta} e^{M(z)-M(z)u^2/s}\,d u \\
    &\quad
    +
    z\int_{|u|\ge \eta} e^{M(z) - M(z) \eta^2/s}
    f(z u+u)\, d u \\
    &
    \le \underbrace{
      \frac{A  e^{M(z)}}{z^\gamma (1-\eta)^{1+\gamma}}
      \sqrt{\frac{\pi s}{M(z)}}
    }_{I_1} + \underbrace{e^{M(z)-M(z) \eta^2/s}}_{I_2}.
  \end{align*}

  Since $e^{M(z)} = z^{\gamma/2}$ and $\sqrt{s} = \sqrt{(1+\eta)}
  \r_n$, for $z\gg 1$, $I_1\le Az^{-\gamma/2} \r_n/2$.  On the other
  hand, $z\gg 1$ and $\r_n \ll 1$, the following
  (in)equalities hold
  \begin{align*}
    I_2 = 
    z^{\gamma(1-\eta^2/s)/2}
    \le z^{-\gamma/2} z^{-\frac{\eta^2}{3(1+\eta)\r_n^2}}
    \le A z^{-\gamma/2} \r_n/2,
  \end{align*}
  so $I_1+I_2 \le A z^{-\gamma/2} \r_n$.  Then
  by \eqref{eq:upper-1}and \eqref{eq:upper-2},
  \begin{align*}
    P\Cbr{\r_n^2 \ge \inf_{z\ge a} \nth n \sum_{i=1}^n \fx{X_i}{z}^2}
    &
    \le
    \sum_{j=0}^\infty \Sbr{A(a\alpha_n^j)^{-\gamma/2}\r_n}^n \\
    &
    = \frac{(A a^{-\gamma/2}\r_n)^n}
    {1-(1+\eta\r_n/4)^{-\gamma n/2}}
  \end{align*}
  
  Since $\r_n\to 0$ and $\Linf_n n\r_n >0$, there is $K>0$ such
  that for all $n\gg 1$, $1- (1+\eta\r_n/4)^{-\gamma n/2}\ge 1-
  e^{-\eta\gamma n\r_n/9} > 1/K$.  Thus
  \begin{align*}
    \nth n \log P\Cbr{
      \r_n^2\ge \inf_{z\ge a} \nth n \sum_{i=1}^n \fx{X_i}{z}^2
    } \le
    \log\r_n + \log (A a^{-\gamma/2}) + \frac{\log K}{n}.
  \end{align*}
  Since $A$ and $K$ are fixed independently of $a$, by choosing
  $a=a(T)$ large enough, \eqref{eq:upper-0} is proved.

  Finally, for $d\in [-\delta,\delta]$ and $z\ge a$,
  \begin{align*}
    \fx{X_i+d}{z}^2 = \frac{(z-d)^2}{z^2} \fx{X_i}{z-d}^2 \ge
    \frac{(a-\delta)^2}{a^2} \fx{X_i}{z-d}^2,
  \end{align*}
  Therefore,
  \begin{align*}
    &\quad
    \sup_{|d|\le \delta}\nth n \log P\Cbr{
      \r_n^2\ge \inf_{z\ge a} \nth n \sum_{i=1}^n \fx{X_i+d}{z}^2
    }
    \\
    &
    \le
    \nth n \log P\Cbr{
      \frac{a^2 \r_n^2}{(a-\delta)^2}
      \ge \inf_{z\ge a-\delta} \nth n \sum_{i=1}^n \fx{X_i}{z}^2
    }.
  \end{align*}
  Then the lemma follows from \eqref{eq:upper-0}.
\end{proof}

\begin{lemma} \rm \label{lemma:lower}
  Suppose $f$ is bounded.  Let $\r_n\to 0$ such that $\Linf_n
  n\r_n>0$.  Then, given $T>0$, there is $b=b(T)>0$, such that for
  $n\gg 1$,
  \begin{align*}
    \sup_{d\in\Reals}\,\nth n  \log P\Cbr{
      \r_n^2\ge \inf_{0<z\le b} \nth n \sum_{i=1}^n \fx{X_i+d}{z}^2
    } \le \log\r_n -T.
  \end{align*}
\end{lemma}
\begin{proof}
  Given $\eta>0$ such that $\eta > (1+\eta/8)^2(1+\eta/4)^2-1$,
  by the same argument for \eqref{eq:upper-1}, for $b>0$,
  $d\in\Reals$ and $n\ge 1$ with $\r_n<1/2$, letting $\alpha_n =
  1+\eta\r_n/4$,
  \begin{align*}
    &\quad
    P\Cbr{\r_n^2 \ge \inf_{z\le b} \nth n \sum_{i=1}^n
      \fx{X_i+d}{z}^2}
    \\
    &\le
    \sum_{j=0}^\infty P\Cbr{
      \r_n^2 \ge \inf_{b\alpha_n^{-j-1} \le z\le b\alpha_n^{-j}} 
      \nth n \sum_{i=1}^n \fx{X_i+d}{z}^2
    } \\
    &\le
    \sum_{j=0}^\infty P\Cbr{
      (1+\eta)\r_n^2 \ge \nth n \sum_{i=1}^n\fx{\alpha_n^j
        (X_i+d)}{b}^2
    }.
  \end{align*}
  
  Denote $A=\sup f$.  For any $s>0$, $z>0$ and $d\in\Reals$,
  \begin{align*}
    \int e^{-(x/z-1)^2/s} f(x-d)\,d x 
    \le A \int e^{-(x/z-1)^2/s} \,d x = A z\sqrt{\pi s}.
  \end{align*}
  Since the density of $X_i+d$ is $f(x-d)$, by Chernoff's inequality,
  \begin{align*}
    P\Cbr{\r_n^2 \ge \inf_{z\le b} \nth n \sum_{i=1}^n
      \fx{X_i+d}{z}^2}
    &\le \sum_{j=0}^\infty
    (A b\alpha_n^{-j} \inline{\sqrt{\pi(1+\eta)}}\r_n)^n \\
    &
    = \frac{(A b\sqrt{\pi(1+\eta)}\r_n)^n}
    {1-(1+\eta\r_n/4)^{-n}}.
  \end{align*}
  By the same argument for Lemma \ref{lemma:upper}, the lemma is then
  proved.
\end{proof}

\begin{lemma}\rm \label{lemma:interval}
  Let $0<b<a<\infty$ and suppose $f$ is continuous and nonzero in a
  neighborhood of $[b,a]$.  If $\r_n\to 0$, then, given $\eta>0$, for
  $n\gg 1$,
  \begin{align*}
    \nth n \log P\Cbr{
      \r_n^2\ge \nth n \sum_{i=1}^n \fx{X_i}{z}^2
    }
    < \log \r_n + \log [\sqrt{2\pi e}\,z f(z)] + \eta,
  \end{align*}
  holds for all $z\in [b,a]$.
\end{lemma}

\begin{proof}
  Fix $c>0$ such that $2\log(1+c)<\eta$.  Because $f$ is continuous
  and positive in a neighborhood of $[b,a]$, there is $s>0$, such that
  for all $z\in [b,a]$ and $u\in [-s, s]$, $f(z+uz) < (1+c)f(z)$.
  Then
  \begin{align*}
    I(z)
    &
    :=
    \int \exp\Cbr{
      \nth 2 - \nth{2\r_n^2} \Grp{\frac{x}{z}-1}^2
    } f(x)\,d x \\
    &
    = z\sqrt{e} \int e^{-u^2/(2\r_n^2)} f(z+u z)\,d u \\
    &
    \le
    z\sqrt{e} (1+c) f(z)\int_{-s}^s e^{-u^2/(2\r_n^2)}\,d u
    +
    z e^{(1-s^2/\r_n^2)/2} \int_{|u|\ge s} f(z+u z)\,d u
    \\
    &
    \le
    \sqrt{2\pi e}\r_n (1+c) z f(z) + e^{(1-s^2/\r_n^2)/2}.
  \end{align*}
  By $\inf_{z\in [b,a]} zf(z)>0$ and $\r_n\to 0$, it follows that for
  $n\gg 1$, $I(z) < \sqrt{2\pi e} \r_n(1+c)^2 z f(z) < \sqrt{2\pi e}
  \r_n e^\eta z f(z)$.  Together with Chernoff's inequality, this
  implies the inequality in the lemma.
\end{proof}

To demonstrate Lemma \ref{lemma:local}, we need the following
application of the uniform exact LDP of \cite{chag:seth:93}.  The
result will be proved in the next section.
\begin{prop} \rm \label{prop:ldp}
  Suppose $f$ is bounded on $\Reals$.  Let $z>0$ such that $f$ is
  continuous and nonzero at $z$.  Define
  \begin{align}
    h(t) = \log\Sbr{z\int e^{-t u^2} f(z+u z)\,d u}, \quad t\ge 0.
    \label{eq:h}
  \end{align}
  Let $\r_n\to 0$ such that $n\r_n^4/\log n\toi$.  Then, for each $n$,
  there is a unique $t_n>0$, such that $h'(t_n) = -\r_n^2$, and
  moreover, as $n\toi$,
  \begin{gather}
    t_n \sim \frac{1}{2\r_n^2}, \label{eq:t}\\
    h(t_n) = \log\r_n + \log[\sqrt{2\pi}\,z f(z)] + o(1)
    \label{eq:sigma-h} \\
    P\Cbr{
      \r_n^2 \ge \nth n \sum_{i=1}^n \fx{X_i}{z}^2
    } \sim
    \frac{\exp\Cbr{n(\r_n^2 t_n + h(t_n))}}{\sqrt{\pi n}}.
    \label{eq:exact}
  \end{gather}
\end{prop}

\SSECT*{\sc Proof of Lemma \ref{lemma:local}}
It suffices to show that there is $\delta=\delta(r)>0$, such that
\begin{align*}
  \lim_{n\toi} \sup_{|d|\le \delta}
  \frac{\displaystyle
    P\Cbr{\r_n^2 \ge \inf_{z>0,\,|z-z_0|>r} \nth n \sum_{i=1}^n
      \fx{X_i+d}{z}^2}
  }{\displaystyle
    P\Cbr{\r_n^2 \ge \inf_{|z-z_0|\le r} \nth n \sum_{i=1}^n
      \fx{X_i+d}{z}^2}
  } = 0.
\end{align*}
Denote the denominator by $B(r,d)$.  Given $0<\eta\ll 1$,
when $|d|\ll \min(r,z_0)$,
\begin{align*}
  B(r,d)&=
  P\Cbr{\r_n^2 \ge \inf_{|z-z_0|\le r} \nth n
    \Grp{\frac{z-d}{z}}^2 \sum_{i=1}^n
    \fx{X_i}{z-d}^2} \\
  &\ge
  P\Cbr{(1-\eta) \r_n^2 \ge \inf_{|z-z_0|\le r-d} \nth n
    \sum_{i=1}^n \fx{X_i}{z}^2} \\
  &\ge
  P\Cbr{(1-\eta) \r_n^2 \ge \nth n
    \sum_{i=1}^n \fx{X_i}{z_0}^2}.
\end{align*}
Therefore, it is enough to show that there is $\delta>0$, such that
\begin{align}
  \lim_{n\toi} 
  \frac{\displaystyle
    \sup_{|d|\le \delta}
    P\Cbr{\r_n^2 \ge \inf_{z>0,\,|z-z_0|>r} \nth n \sum_{i=1}^n
      \fx{X_i+d}{z}^2} 
  }{\displaystyle
    P\Cbr{(1-\eta) \r_n^2 \ge \nth n\sum_{i=1}^n \fx{X_i}{z_0}^2}
  } = 0.  \label{eq:local-1}
\end{align}

By the assumption of the lemma,
\begin{align*}
  D := \log [z_0 f(z_0)] - \sup_{z>0,\,|z-z_0|\ge r/2} \log [z f(z)]
  > 0.
\end{align*}

By Proposition \ref{prop:ldp}, as long as $\eta>0$ is small
enough, as $n\to 0$,
\begin{align}
  &\quad
  \nth n \log 
  P\Cbr{(1-\eta) \r_n^2 \ge \nth n\sum_{i=1}^n \fx{X_i}{z_0}^2}
  \nonumber \\
  &
  =(1-\eta) \r_n^2 t_n + \log (\sqrt{1-\eta}\,\r_n) + \log [
  \sqrt{2\pi}\, z f(z)] + o(1) \nonumber \\
  &
  \ge M_n :=\log\r_n + \log [\sqrt{2\pi e} z_0 f(z_0)] -
  D/4.
  \label{eq:local-2}
\end{align}

Since $\r_n\to 0$ and $n\r_n^4/\log n\toi$, $n\r_n\toi$ as well.  By
Lemmas \ref{lemma:upper} -- \ref{lemma:lower}, there are $b\in
(0,z_0-r)$, $a\in (z_0+r,\infty)$ and $\delta_0>0$, such that 
\begin{gather}
  \sup_{|d|\le \delta_0}\nth n \log P\Cbr{
    \r_n^2\ge \inf_{z\not\in[b,a]} \nth n \sum_{i=1}^n
    \fx{X_i+d}{z}^2
  } \le M_n - D/2,  \label{eq:local-3}
\end{gather}

Fix $0<\delta \le \delta_0$ such that $\delta < \min(r/2, b/2, a)$
and $z^2 < (1+\eta)(z-\delta)^2$ for all $z\in [b,a]$.  Then
\begin{align*}
  &\quad
  \sup_{|d|\le \delta}
  P\Cbr{
    \r_n^2 \ge \inf_{b\le z\le a,\,|z-z_0|>r}
    \nth n \sum_{i=1}^n \fx{X_i+d}{z}^2
  }
  \\
  &=
  \sup_{|d|\le \delta}
  P\Cbr{
    \r_n^2 \ge \inf_{b\le z\le a,\,|z-z_0|>r}
    \nth n \Grp{\frac{z-d}{z}}^2 \sum_{i=1}^n \fx{X_i}{z-d}^2
  } \\
  &\le
  P\Cbr{
    \sup_{b\le z\le a,\,|d|\le \delta} \Grp{\frac{z}{z-d}}^2 \r_n^2
    \ge \inf_{b-\delta\le z \le a+\delta,\, |z-z_0|\ge r/2} \nth n 
    \sum_{i=1}^n \fx{X_i}{z}^2
  } \\
  &\le
  P\Cbr{
    (1+\eta)\r_n^2 \ge \inf_{z\in J} \nth n 
    \sum_{i=1}^n \fx{X_i}{z}^2
  },
\end{align*}
where $J = [b/2, z_0-r/2] \cup [z_0+r/2, 2a]$.  By Lemma
\ref{lemma:interval} and the definition of $D$, as long as $\eta$ is
chosen small enough, for all $n\gg 1$,
\begin{gather*}
  \sup_{z\in J}
  \nth n \log P\Cbr{
    (1+\eta)^4\r_n^2\ge \nth n \sum_{i=1}^n \fx{X_i}{z}^2
  } \le M_n - D/2.
\end{gather*}

Let $\alpha_n = 1+\eta\r_n$ and $N(n) = \lceil\log(4a/b) /
\log\alpha_n\rceil$.  It is not hard to see that $J$ can be
covered by the union of at most $N(n)$ intervals of the form $I_k
= [x_k, \alpha_n x_k]$.  By Lemma \ref{lemma:scale} and the above
inequality, for $n\gg 1$,
\begin{align*}
  &\quad
  P\Cbr{
    (1+\eta)\r_n^2 \ge \inf_{z\in J} \nth n 
    \sum_{i=1}^n \fx{X_i}{z}^2
  } \\
  &
  \le
  \sum_k
  P\Cbr{
    (1+\eta)\r_n^2 \ge \inf_{z\in I_k} \nth n 
    \sum_{i=1}^n \fx{X_i}{z}^2
  } \\
  &
  \le
  N(n) \max_k 
  P\Cbr{
    (1+\eta)^4\r_n^2 \ge
    \nth n \sum_{i=1}^n \fx{X_i}{\alpha_n x_k}^2
  }
\end{align*}
and hence
\begin{align}
  &\quad
  \sup_{|d|\le \delta} \nth n \log
  P\Cbr{
    \r_n^2 \ge \inf_{b\le z\le a,
      \,|z-z_0|>r} \nth n \sum_{i=1}^n \fx{X_i+d}{z}^2
  }
  \nonumber\\
  &
  \le M_n - \frac{D}{2} + \frac{\log N(n)}{n}.  \label{eq:local-5}
\end{align}

As $n\toi$, $N(n) \sim \log(4a/b)/(\eta\r_n) = O(\r_n^{-1})$.
Since $n\r_n^4/\log n\toi$, $\log N(n)/n \to 0$.  By combining
\eqref{eq:local-2} -- \eqref{eq:local-5}, \eqref{eq:local-1} is thus
proved.  \qed

\SECT{Proof of Proposition \ref{prop:ldp}}
Given $z>0$, the log-moment generating function of $-(X/z-1)^2$ is
$h(t)$, which is defined in \eqref{eq:h}.  It is not hard to see that
for $t\ge 0$, $h(t)<\infty$ and
\begin{align}
  h'(t) &
  =
  -\frac{\int u^2 e^{-t u^2} f(z+u z)\,d u}{
    \int e^{-t u^2} f(z+u z)\,d u} < 0, \label{eq:h-diff} \\
  h''(t) &
  =
  \frac{\int u^4 e^{-t u^2} f(z+u z)\,d u}{
    \int e^{-t u^2} f(z+u z)\,d u} - [h'(t)]^2>0.  \label{eq:h-diff2}
\end{align}
\begin{lemma} \rm \label{lemma:simple}
  Fix $z>0$.  Suppose $f$ is continuous and nonzero at $z$.  Also
  suppose $\sup f < \infty$.  Then for $p>-1$,
  \begin{align*}
    \int |u|^p e^{-t u^2} f(z+u z)\,d u \sim f(z) \Gamma(p') t^{-p'},
    \quad \text{as}\ \ t\toi,
  \end{align*}
  with $p' = (p+1)/2$.
\end{lemma}
\begin{proof}
  Given $c>1$, there is $\eta>0$, such that $f(z)/c < f(z+uz) < c
  f(z)$ for all $u\in [-\eta, \eta]$.  Write
  \begin{align*}
    \int |u|^p e^{-t u^2} f(z+u z)\,d u 
    = \int_{-\eta}^\eta  + \int_{|u|>\eta} = I_1 + I_2.
  \end{align*}
  By the selection of $c$,
  \begin{align*}
    \frac{f(z)}{c} \int_{-\eta}^\eta |u|^p e^{-t u^2} \,d u < I_1
    <    c f(z)\int_{-\eta}^\eta |u|^p e^{-t u^2} \,d u.
  \end{align*}
  As $t\toi$,
  \begin{gather*}
    \int_{-\eta}^\eta |u|^p e^{-t u^2}\,d u \sim \int
    |u|^p e^{-t u^2}\,d u = \Gamma(p') t^{-p'}, \\
    I_2 \le \sup f \int_{|u|>\eta} |u|^p e^{-t u^2} \,d u = o(t^{-p'}).
  \end{gather*}
  Since $c>1$ is arbitrary, the lemma is proved.
\end{proof}

\SSECT*{\sc Proof of Proposition \ref{prop:ldp}}
Because $h''>0$ on $(0,\infty)$, $h'$ is strictly increasing on
$(0,\infty)$.  By Lemma \ref{lemma:simple}, $h'(t)\sim -(2t)^{-1}$
as $t\toi$.  Thus, by $\r_n\to 0$, for $n\gg 1$, there is a
unique $t_n\toi$ with $\r_n^2 = -h'(t_n) \sim (2t_n)^{-1}$.  This
proves \eqref{eq:t}.  By Lemma \ref{lemma:simple},
\begin{align*}
  h(t_n) 
  = \log \Sbr{ z\int e^{-t_n u^2} f(z+u z)\,d u}
  = \log [\,(1+o(1))z f(z) \mbox{$\sqrt{\pi/t_n}$}\,\,].
\end{align*}
Together with \eqref{eq:t}, this implies \eqref{eq:sigma-h}.

It remains to show \eqref{eq:exact}.  For large $n$, $t_n$ is
well-defined.  Because $\r_n^2 t + h(t)$ is strictly convex, $t_n =
\arg\inf_{t>0} [\r_n^2 t + h(t)]$.  Let
\begin{align*}
  f_n(x) = e^{-t_n(x/z-1)^2 - h(t_n)} f(x).
\end{align*}
It is seen that $f_n$ is a probability density.  Let $\xi_{nk} =
-(\zeta_{nk}/z-1)^2$, where $\zeta_{nk}$ are iid with density $f_n$.
Then by \eqref{eq:h-diff}
\begin{align*}
  E(\xi_{nk}) = -\int (x/z-1)^2 f_n(x) 
  = - z \int u^2 e^{-t_n u^2-h(t_n)} f(z+z u) \,du = h'(t_n)
\end{align*}
and likewise by \eqref{eq:h-diff2}, $\mathrm{Var}(\xi_{nk}) =
h''(t_n)$.  Define
\begin{gather}
  \begin{cases}
    \displaystyle
    Y_n = \frac{\xi_{n1} + \ldots + \xi_{n n} - n h'(t_n)}{\sqrt{n
        h''(t_n)}}, \\[2.5ex]
    \displaystyle
    T_n = -\sum_{i=1}^n \fx{X_i}{z}^2, \ \
    G_n(t) = E[e^{t T_n}], \ \ t>0
  \end{cases} \label{eq:ytg}
\end{gather}
and $\Lambda_n(t) = \log G_n(t)$.  By checking the characteristic
function of $\xi_{n1} + \ldots + \xi_{n n}$, it can be seen that
$Y_n$ also has the representation
\begin{gather}  \label{eq:tilt}
  Y_n \sim \frac{\tilde T_n -
    \Lambda_n'(t_n)}{\sqrt{\Lambda_n''(t_n)}}, \ 
  \text{ with }\ 
  P(\tilde T_n\in d x) = e^{t_n x - \Lambda_n(t_n)} P(T_n\in d x), 
\end{gather}
and hence characteristic function
\begin{gather}
  E[e^{itY_n}] = \exp\Cbr{-
    \frac{i t \Lambda_n'(t_n)}{\sqrt{\Lambda_n''(t_n)}
    }
  } G_n\Grp{t_n + \frac{it}{\sqrt{\Lambda_n''(t_n)}}} \bigg/
  G_n(t_n).  \label{eq:y-cf}
\end{gather}

Since $\Lambda_n(t) = n h(t)$, then $\Lambda_n'(t_n) = -n\r_n^2$ and,
by Lemma \ref{lemma:simple} and \eqref{eq:t},
\begin{align}
  \Lambda_n''(t_n) = n h_n''(t_n) \sim n/(2t_n^2) \sim 2 n\r_n^4,
  \quad\text{as}\ \ n\toi.   \label{eq:tilt-var}
\end{align}
By standard exponential tilting,
\begin{align*}
  &\hspace{-2ex}
  P\Cbr{\r_n^2 \ge \nth n \sum_{i=1}^n \fx{X_i}{z}^2}
  = P\Cbr{T_n \ge -n\r_n^2} \\
  &=
  e^{n\r_n^2 t_n + \Lambda_n(t_n)} E\Sbr{
    \cf{Y_n\ge 0} e^{-t_n\sqrt{\Lambda_n''(t_n)} Y_n}}
\end{align*}
Therefore, in order to show \eqref{eq:exact}, it suffices to show
\begin{align}
  E\Cbr{\cf{Y_n\ge 0} e^{-t_n\sqrt{\Lambda_n''(t_n)}Y_n}}
  \sim \nth{t_n\sqrt{2\pi\Lambda_n''(t_n)}}.
  \label{eq:cs-limit}
\end{align}
The proof is based on the next lemma, which is essentially established
in \cite{chag:seth:93}. 
\begin{lemma}\rm \label{lemma:cs}
  For each $n$, let $T_n$ be a random variable such that $G_n(t) =
  E[e^{tT_n}]<\infty$ in a neighborhood of $t_n\in\Reals$.  Let
  $\Lambda_n(t) = \log G_n(t)$ and $Y_n$ be defined as in
  \eqref{eq:tilt}.  Suppose that, as $n\toi$,
  \begin{gather}
    \Lambda''_n(t_n) \toi, \quad t_n^2 \Lambda''_n(t_n)\toi,
    \label{eq:cs-basic} \\
    Y_n \stackrel{d}{\to} N(0,1), \label{eq:cs-wc}
  \end{gather}
  and there is $\delta>0$ and $n_0\ge 1$, such that
  \begin{gather}
    f^*(t) := \sup_{n\ge n_0} \Abs{E[e^{it Y_n}]
      \cf{|t|\le \delta \inline{\sqrt{\Lambda''_n(t_n)}}}} \in L^1,
    \label{eq:cs-l1} \\
    \sup_{\delta < |y| \le \lambda t_n}
    \Abs{
      \frac{G_n(t_n+i y)}{G_n(t_n)}
    } = o\Grp{\nth{t_n\sqrt{\Lambda''_n(t_n)}}}, \forall\lambda>0.
    \label{eq:cs-mgf}
  \end{gather}
  Then \eqref{eq:cs-limit} holds.
\end{lemma}
\begin{proof}
  Let $\beta_n = \delta \sqrt{\Lambda''_n(t_n)}$ and $b_n =
  t_n\sqrt{\Lambda''_n(t_n)}$.  Then by \eqref{eq:y-cf}, the
  characteristic function of $Y_n$ satisfies conditions (2.7) and
  (2.8) of Theorem 2.3 in \cite{chag:seth:93}, and hence (2.9) and
  (2.10) there.  Then by $Y_n\to N(0,1)$ and Theorem 2.7 in
  \cite{chag:seth:93}, \eqref{eq:cs-limit} follows.
\end{proof}

Continuing the proof of Proposition \ref{prop:ldp}, it suffices to
verify \eqref{eq:cs-basic} -- \eqref{eq:cs-mgf} for $T_n$ defined in
\eqref{eq:ytg}.  By \eqref{eq:tilt-var} and the assumption that
$n\r_n^4/\log n\toi$, \eqref{eq:cs-basic} is clear.  To show
\eqref{eq:cs-wc}, consider the representation in \eqref{eq:ytg}.
Because $EY_n=0$ and $\mathrm{Var}(Y_n)=1$, we only need to check
the Lindeberg condition, i.e., for any $a>0$,
\begin{align*}
  n E\Sbr{
    \Grp{\frac{\xi_n-h'(t_n)}{\sqrt{n h''(t_n)}}}^2 
    \cf{
      \Abs{\frac{\xi_n-h'(t_n)}{\sqrt{n h''(t_n)}}} \ge a
    }
  }\to 0, \quad\text{with}\ \xi_n\sim\xi_{n k}.
\end{align*}

Since $h'(t_n)=\r_n^2$ and $h''(t_n) \sim 2\r_n^4$, for $n\gg 1$,
$|\xi_n - \r_n^2|\ge a \sqrt{n h''(t_n)}$ implies $|\xi_n|\ge
a\sqrt{n} \r_n^2$ and $|\xi_n - \r_n^2| \le 2 |\xi_n|$.  It thus
suffices to show
\begin{align} \label{eq:cs-lind}
  E\Sbr{\xi_n^2 \cf{|\xi_n|\ge a\sqrt{n}\r_n^2}} = o(\r_n^4).
\end{align}

By the definition of $f_n$, the expectation on the left hand side is
equal to 
\begin{align*}
  &
  e^{-h(t_n)} \int_{(x/z-1)^2\ge a\sqrt{n}\r_n^2} \fx{x}{z}^4
  e^{-t_n(x/z-1)^2} f(x)\,d x \\
  =\ &
  z e^{-h(t_n)} \int_{u^2\ge a\sqrt{n}\r_n^2} u^4 e^{-t_n u^2}
  f(z+z u)\,d u \le\ (z \sup f) e^{-h(t_n)}I_n,
\end{align*}
where, by change of variable $u = x/\sqrt{t_n}$ and $\r_n^2 t_n \sim
1/2$,
\begin{align*}
  I_n:=\int_{u^2\ge a\sqrt{n}\r_n^2} u^4 e^{-t_n u^2}\,d u 
  \le t_n^{-5/2} \int_{x^2\ge b\sqrt{n}} x^4 e^{-x^2} \,d x,
  \quad\text{for}\ n\gg 1,
\end{align*}
with $b\in (0,a/2)$ a constant.  Since
\begin{align*}
  \int_{x^2\ge b\sqrt{n}} x^4 e^{-x^2} \,d x
  = \int_{b\sqrt{n}}^\infty y^{3/2} e^{-y}\,d y \sim (b\sqrt{n})^{3/2}
  e^{-b\sqrt{n}} = o(1),
\end{align*}
$I_n = o(t_n^{-5/2}) = o(\r_n^5)$.  On the other hand, by Lemma
\ref{lemma:simple}, $e^{-h(t_n)} \sim 1/(zf(z) \sqrt{2\pi} \r_n) =
O(\r_n^{-1})$.  As a result, $e^{-h(t_n)} I_n = o(\r_n^4)$, yielding
\eqref{eq:cs-lind}.

To show \eqref{eq:cs-l1} and \eqref{eq:cs-mgf}, notice that 
\begin{align*}
  &
  \Abs{\frac{G_n(t_n + i y)}{G_n(t_n)}} = |\phi_n(y)|^n \\
  \text{where}\quad
  &
  \phi_n(y) = z e^{-h(t_n)}
  \int e^{-t_n u^2 - i y u^2} f(z+z u)\,d u.
\end{align*}

Fix $0<c\ll 1$.  Since $f$ is continuous and nonzero at $z$, there
is $r\in (0,1/2)$ such that $f(z)/(1+c) \le f(z+uz)$ and $1-u^2 \le
\cos u \le 1-u^2/(2+c)$ for $u\in [-r, r]$.  Write 
\begin{align*}
  \phi_n(y)=
  z e^{-h(t_n)}\int_{|u|\le  \sqrt{r/y}} +
  z e^{-h(t_n)}\int_{|u|> \sqrt{r/y}} = I_n(y) +
  J_n(y).
\end{align*}
Then for $n\gg 1$ and $y\in\Reals$, 
\begin{align*}
  \Abs{\mathrm{Re}\,I_n(y)}
  &= z e^{-h(t_n)} \int_{-\sqrt{r/y}}^{\sqrt{r/y}}
  \cos(y u^2) e^{-t_n u^2} f(z+z u)\,d u
  \\
  &
  \le
  1- \frac{z e^{-h(t_n)}}{2+c}
  \int_{-\sqrt{r/y}}^{\sqrt{r/y}} y^2 u^4 e^{-t_n u^2} f(z+z u)\,d u \\
  &
  \le
  1 - \frac{z e^{-h(t_n)} y^2}{2+c} 
  \int_{-\sqrt{r/(y\vee 1)}}^{\sqrt{r/(y\vee 1)}}
  u^4 e^{-t_n u^2} f(z+z u) \,d u
  \\
  &
  \le
  1 - \frac{z e^{-h(t_n)} y^2 f(z)}{(1+c)(2+c)} 
  \int_{-\sqrt{r/(y\vee 1)}}^{\sqrt{r/(y\vee 1)}} u^4 e^{-t_n u^2}
  \,d u
  \\
  &\le
  1 - \frac{y^2 \r_n^4}{2(1+2c)} \nth{\sqrt{2\pi}}
  \int_{-\sqrt{2t_n r/(y\vee 1)}}^{\sqrt{2t_n r/(y\vee 1)}}
  u^4 e^{-u^2/2}\,d u,
\end{align*}
where the last inequality is due to change of variable, 
\eqref{eq:t} and \eqref{eq:sigma-h}.  Since $t_n\toi$, by choosing
$M\gg 1/r$, for $n\gg 1$ and $|y|\le (r/M) t_n$,
\begin{align*}
  |\mathrm{Re}\,I_n(y)| \le  1 - \frac{3 y^2 \r_n^4}{2(1+3c)}.
\end{align*}

On the other hand, by Lemma \ref{lemma:simple}, \eqref{eq:t} and
\eqref{eq:sigma-h}, for $n\gg 1$ and $y\in\Reals$,
\begin{align*}
  \Abs{\mathrm{Im}\,I_n(y)}
  &
  \le z e^{-h(t_n)}
  \int_{-\sqrt{r/y}}^{\sqrt{r/y}} |\sin(y u^2)| e^{-t_n u^2}
  f(z+z u)\,d u
  \\
  &
  \le |y| z e^{-h(t_n)} 
  \int u^2 e^{-t_n u^2} f(z+z u)\,d u \\
  &
  \le \sqrt{1+c}\,|y|\r_n^2\,.
\end{align*}
As a result, for $n\gg 1$ and $|y|\le (r/M)t_n$,
\begin{align}
  |I_n(y)|
  \le \sqrt{\Sbr{1-\frac{3y^2\r_n^4}{2(1+3c)}}^2 + (1+c) y^2\r_n^4}
  \le 1-\frac{y^2 \r_n^4}{1+c}, \label{eq:cs-I}
\end{align}

On the other hand, for $n\gg 1$,
\begin{align*}
  |J_n(y)|
  &\le z e^{-h(t_n)} \sup f \int_{|u|\ge \sqrt{r/y}} e^{-t_n
    u^2} \,d u \\
  &\le
  (1+c) A P(|Z|\ge \inline{\sqrt{2 t_n r/y}}),
\end{align*}
where $Z\sim N(0,1)$ and $A=\sup f/f(z)<\infty$.  Recall that
$P(|Z|\ge x) \sim \sqrt{2/\pi} x^{-1} e^{-x^2/2}$ as $x\toi$.
Therefore, for $M\gg 1/r$ and $|y|\le (r/M)t_n$, $|J_n(y)| \le e^{-t_n
  r/y}\le (y/t_n)^2/20 \le y^2\r_n^4/4$.

Combining the bounds for $I_n(y)$ and $J_n(y)$,
\begin{align}
  |\phi_n(y)| \le 1-\Grp{\nth{1+c}- \nth 4}
  y^2\r_n^4 \le e^{-y^2\r_n^4/2},
  \quad
  |y|\le (r/M)t_n.
  \label{eq:phi}
\end{align}

To verify \eqref{eq:cs-l1} holds for any $\delta>0$ and
$n_0=n_0(\delta)\gg 1$, by \eqref{eq:tilt},
\begin{align*}
  \Abs{E[e^{it Y_n}]}
  =
  \Abs{\frac{G_n(t_n + it/\sqrt{\Lambda_n''(t_n)})}{G_n(t_n)}}
  =
  \Abs{\phi_n\Grp{\frac{t}{\sqrt{\Lambda_n''(t_n)}}}}^n.
\end{align*}
Then, letting $y =t/\!\sqrt{\Lambda_n''(t_n)}$, by \eqref{eq:phi},
for $n\gg 1$ such that $(r/M)t_n \ge \delta$, 
\begin{align*}
  \Abs{
    E[e^{it Y_n}]
    \cf{|t|\le \delta\!\inline{\sqrt{\Lambda_n''(t_n)}}}
  } 
  =
  \Abs{\phi_n(y)}^n \cf{|y|\le\delta} \le 
  e^{-n y^2\r_n^4/2}.
\end{align*}
By \eqref{eq:tilt-var}, the right hand side is no greater than
$e^{-t^2/9}$, which proves \eqref{eq:cs-l1}. 

To verify \eqref{eq:cs-mgf}, fix $\delta>0$ and first let $\lambda \le
r/M$.  Then by \eqref{eq:phi},
\begin{align*}
  \sup_{\delta \le |y| \le \lambda t_n}
  \Abs{\frac{G_n(t_n+i y)}{G_n(t_n)}} =
  \sup_{\delta \le |y| \le \lambda t_n}|\phi_n(y)|^n
  \le e^{-\delta^2 n \r_n^4/2}.
\end{align*}
Since $n\r_n^4/\log n\toi$, the right hand side is $o(1/\sqrt{n})$.
On the other hand, by \eqref{eq:tilt-var}, $t_n
\inline{\sqrt{\Lambda_n''(t_n)}} \sim \sqrt{n/2}$.  Thus
\eqref{eq:cs-mgf} holds.

Finally, let $\lambda > \eta:= r/M$.  From the above proof, it
suffices to bound
\begin{align*}
  \sup_{\eta t_n\le |y|\le \lambda t_n}
  \Abs{
    \frac{G_n(t_n+i y)}{G_n(t_n)}
  }
  =
  \sup_{\eta t_n\le |y|\le \lambda t_n}
  \Abs{
    \frac{\int e^{-(t_n+i y)u^2} f(z+z u)\,d u}{
      \int e^{-t_n u^2} f(z+z u)\,d u
    }
  }^n.
\end{align*}

By change of variable $u = x/\sqrt{2t_n}$ and letting $\theta =
y/t_n$,
\begin{align*}
  &
  \frac{\int e^{-(t_n+i y)u^2} f(z+z u)\,d u}{
    \int e^{-t_n u^2} f(z+z u)\,d u
  }
  = \int e^{-i\theta x^2/2} g_n(x)\,d x, \\
  \text{where}\quad 
  &
  g_n(x) = \frac{e^{-x^2/2} f(z + z x/\sqrt{2t_n})}{
    \displaystyle
    \int e^{-x^2/2} f(z + z x/\sqrt{2t_n})\,d x
  }.
\end{align*}
For $y\in\Reals$ with $|y|\le \lambda t_n$, $\theta\in [-\lambda/2,
  \lambda/2]$.  By the 
continuity of $f$ at $z$ and $f(z)>0$, $g_n(x) \to
e^{-x^2/2}/\sqrt{2\pi}$ pointwise.  So by dominated convergence 
\begin{align*}
  \int e^{-i\theta x^2/2} g_n(x)\,d x \to \frac{1}{\sqrt{1+i\theta}}
\end{align*}
uniformly for $\theta\in [-\lambda/2, \lambda/2]$.  Given $c>1$,
for all $n\gg 1$,
\begin{align*}
  \Abs{\int e^{-i\theta x^2/2} g_n(x)\,d x} \le
  \frac{c}{|\sqrt{1+i\theta}|} =  \frac{c}{(1+\theta^2)^{1/4}}.
\end{align*}
It follows that
\begin{align*}
  \sup_{\eta t_n\le|y|\le \lambda t_n}
  \Abs{
    \frac{G_n(t_n+i y)}{G_n(t_n)}
  }
  \le
  c^n \Grp{1+\frac{\eta^2}{4}}^{-n/4}.
\end{align*}
By choosing $c\approx 1$, the right hand side is $\alpha^n$ for some
$\alpha \in (0,1)$, and hence is $o(1/(t_n\sqrt{\Lambda_n''(t_n)}))$.
The entire \eqref{eq:cs-mgf} is thus verified.  \qed

\SECT*{Appendix}
To prove \eqref{eq:sphere}, let $e_1, \ldots, e_n$ be the standard
basis of $\Reals^n$.  Let $u_0 = (1/\sqrt{n})\sum_{i=1}^n e_i$,
$u_1, \ldots, u_{n-1}\in\Reals^n$ be an 
orthonormal basis.  Under $\{u_i\}$, the coordinates of $\sum_{i=1}^n
X_i e_i$ are $Y_0, Y_1, \ldots, Y_{n-1}$, with $Y_0=\sqrt{n}\bar X$.
Then $\bar X$ and $Y=(Y_1, \ldots, Y_{n-1})$ have joint density
\begin{align*}
  g(t,y) = \sqrt{n} \prod_{i=1}^n f(x_i),\quad
  \text{with }\ 
  \sum_{i=1}^n x_i e_i = \sqrt{n} t u_0 + \sum_{i=1}^{n-1} y_i
  u_i, \ y\in \Reals^{n-1}.
\end{align*}

On the other hand, $V\sim |Y|/\sqrt{n}$ and $\xi := Y/|Y|\in B_{n-1} =
\{x\in \Reals^{n-1}: |x|=1\}$ almost surely, where $|\cdot|$ stands
for the $L^2$-norm. Let $\nu$ be the uniform measure on $B_{n-1}$.  By
$Y = \sqrt{n} V \xi$, $g(t, y)$ and the joint density $k(t,s,z)$ of
$(\bar X, V, \xi)$ with respect to $dt\,ds\, \nu(dz)$ are related via
\begin{align*}
  g(t, y) = \frac{k(t,s,z)}{(\sqrt{n})^{n-1}
    s^{n-2}}, \quad\text{with }\ y = \sqrt{n} s z.
\end{align*}

Since $\phi: z\to \omega = \sum_{i=1}^{n-1} z_i u_i$ is an
isometric mapping from $B_{n-1}$ to $U_n$, $\phi^*\nu$ is the uniform
measure on $U_n$.  Eq.~\eqref{eq:sphere} then follows from
\begin{align*}
  h(t,s)
  &
  = \int k(t,s,z)\,\nu(d z) \\
  &
  =
  (\sqrt{n})^{n-1} s^{n-2} \int g(t, \sqrt{n} s z)\,\nu(d z)
  \\
  &
  =
  (\sqrt{n})^n s^{n-2} \int_{B_{n-1}} \prod_{i=1}^n f(t+\sqrt{n} s
  \omega_i) \,\nu(d z) \\
  &
  =
  (\sqrt{n})^n s^{n-2} \int_{U_n} \prod_{i=1}^n f(t+\sqrt{n} s
  \omega_i) \,(\phi^*\nu)(d\omega).
\end{align*}

\bibliographystyle{acmtrans-ims}

\end{document}